\RequirePackage{fix-cm}
\documentclass{amsart}
\usepackage{graphicx}
\usepackage{amsmath} %
\usepackage{amssymb} %
\usepackage{amsfonts}%

\newtheorem{defin}{{\bf Definition}}[section]
\newtheorem{theorem}[defin]{{\bf Theorem}}

\newtheorem{lemma}[defin]{{\bf Lemma}}

\newtheorem{proposition}[defin]{{\bf Proposition}}

\newtheorem{corollary}[defin]{{\bf Corollary}}

\newtheorem{example}[defin]{\noindent {\bf Example}}
\newtheorem{problem}[defin]{\noindent {\bf Problem}}

\newcommand{\CA}{{\mathcal A}}
\newcommand{\CB}{{\mathcal B}}
\newcommand{\CM}{{\mathcal M}}

\newcommand{\BQ}{{\mathbb Q}}
\newcommand{\BR}{{\mathbb R}}
\newcommand{\BZ}{{\mathbb Z}}
\newcommand{\CC}{2^{\aleph_0}}

\newcommand{\sib}{sib}

\def\powerset #1 { {\mathcal P} ( #1 ) }

\title{Equimorphy -- The Case of Chains}

\author[C. Laflamme, R.Woodrow] {Claude Laflamme*, R. Woodrow}
\thanks{*Supported by NSERC of Canada Grant \# 690404}
\address{University of Calgary, Department of Mathematics and Statistics, Calgary, Alberta, Canada T2N 1N4} 
\email {laflamme@ucalgary.ca} 
\email {woodrow@ucalgary.ca} 

\author[M.~Pouzet]{M.~Pouzet}
\address{ICJ, Math\'ematiques, Universit\'e Claude-Bernard Lyon 1 \\ 43 Bd. 11 Novembre 1918 \\ F$69622$ Villeurbanne  cedex, France }
\email{maurice.pouzet@univ-lyon1.fr}

\date{July 10, 2014}

\begin{document}

\begin{abstract}
Two structures are said to be \emph{equimorphic} if each embeds in the
other. Such structures cannot be expected to be isomorphic, and in this
paper we investigate the special case of linear orders, here also called
chains. In particular we provide structure results for chains having
less than continuum many isomorphism classes of equimorphic chains. We
deduce as a corollary that any chain has either a single isomorphism
class of equimorphic chains or infinitely many.
\end{abstract}

\maketitle

\begin{center}
\small{Dedicated to James E. Baumgartner, for his warmth and inspirations}
\end{center}

\section{Introduction}\label{intro}

Two structures are called \emph{equimorphic} (see Fra\"{\i}ss\'e
\cite{key-F}) if each embeds in the other. Generally one cannot expect
equimorphic structures to necessarily be isomorphic. However the
famous Cantor-Bernstein-Schroeder Theorem states that this is the case
for structures in a language with pure equality: if there is an
injection from one set to another and vice-versa, then there is a
bijection between these two sets. The same situation occurs in other
structures such as vectors spaces, where embeddings are linear
injective maps. But as expected it is not in general the case that
equimorphic structures are isomorphic.

In this paper we study the case of the language of one binary
predicate interpreted as a linear order, also called chains, and show
that the situation here is already quite complex. For example one
readily sees that the rationals together with a largest point added is
a chain equimorphic to the rationals themselves, but certainly not
isomorphic as linear orders. In fact we show that for each cardinal
$\kappa$ there is a chain with exactly $\kappa$ isomorphism classes of
equimorphic chains.  We further provide structure results for chains
having less than continuum isomorphism classes of equimorphic chains,
and deduce as a corollary that any chain has either a single
isomorphism class of equimorphic chains or infinitely many. 

In \cite{key-TH}, Thomass\'e conjectures that any countable relational
structure has either a single isomorphism class of equimorphic
structures, countably many, or else continuum many. We verify  this
conjecture for the case of chains.

We conclude this section with some basic terminology used in this
paper. In general an embedding from a structure $\CA$ to a structure
$\CB$ in the same language is an injective map from $\CA$ to $\CB$
which preserves the given structure; in the case of linear orders
$\CA=\langle A;<_A \rangle$ and $\CB=\langle B;<_B \rangle$, this mean
an injective map $f:A \rightarrow B$ such that $x <_A y$ if and only
if $f(x) <_B f(y)$ for all $x,y \in A$.

We will write $\CA^*$ for the reverse order $\CA^*=\langle A;>_A
\rangle$ obtained by reversing the order, and $\CA+\CB$ for the linear
order obtained by extending the two orderings imposing that every
element of $A$ precedes every element of $B$. Given structures $\CA$
and $\CB$, we write $\CA \leq \CB$ if there is an embedding from $\CA$
to $\CB$, and write $\CA \equiv \CB$ if both $\CA \leq \CB$ and $\CB
\leq \CA$; in this case we say that $\CA$ and $\CB$ are
\emph{equimorphic}, or that $\CB$ is a
\emph{sibling} of $\CA$ (and vice-versa). Note that Bonato et
al. \cite{key-BBDS} refer to such a $\CB$ as a \emph{twin} if moreover
it is not isomorphic to $\CA$.  We shall be interested in describing
these siblings and counting their number but obviously only up to
isomorphism, which we denote by $\sib(\CA)$. Thus $\sib(\CA)=1$ means
that all siblings are isomorphic to $\CA$, or that $\CA$ has no twins.
Note also that obviously $\sib(\CA)=\sib(\CA^*)$ for any chain $\CA$.
We further write as usual $\CA \cong \CB$ when the two structures are
isomorphic.

Other more standard terminology and notation can be found in the books
by Fra\"{\i}ss\'e \cite{key-F} and Rosenstein
\cite{key-R}. In particular we assume the reader to be generally
familiar with the notion of indecomposability and Hausdorff rank of a
linear order although we briefly review these notions below.

\section{Ordinals, sums of ordinals and reverse ordinals}
\label{sec:ord}

It is easy to see that any ordinal has only one sibling up to
isomorphism, that is $\sib(\lambda)=1$ for any ordinal $\lambda$. More
generally this is the case for any finite sum of ordinals or reverse
ordinals.

\begin{proposition}[Finite sums of ordinals and reverse ordinals]\label{prop:finitesumord}
\noindent If $C$ is a finite sums of ordinals and reverse ordinals, then $\sib(C)=1$.
\end{proposition}

\begin{proof}
Let $n$ be the least integer such that $C$ has a decomposition as a
sum of $n$ ordinals or reverse ordinals.  Choose a decomposition $C:=
\sum_{i<n} C_i$ minimal in the sense that if $C:= \sum_{i<n} C'_i$ is
an other decomposition with $C'_i\leq C_i$ for $i<n$ then $C'_i$ is
equimorphic to $C_i$ for all $i<n$; this exists since ordinals are
well ordered under embeddability. 

Now consider any chain $C'\equiv C$. Since $C'\leq C$, $C'$ must be of
the form $C':= \sum_{i<n} C'_i$ with $C'_i\leq C_i$ for all
$i<n$. Since $C\leq C'$, the same argument yields that $C:= \sum_{i<n}
C''_i$ with $C''_i\leq C'_i$ for all $i<n$. Since $C'_i\leq C_i$ we
have $C''_i\leq C_i$. From the minimality of the decomposition of $C$,
we have $C''_i \equiv C_i$ hence $C'_i \equiv C_i$. This yields
$C'_i \simeq C_i$ thus $C' \simeq C$.
\end{proof}

Note that an ordinal chain for example is rigid, meaning it has no
non-trivial automorphisms. On the other hand it has non-trivial
embeddings, and this subtle distinction will soon play a role.  The
situation very much differs with infinite sums of ordinals as the
following example shows, and this allows us to easily find chains with
any prescribed value of siblings, in particular $\sib(\lambda \cdot \omega^*) ( =
 \sib(\lambda^* \cdot \omega )) = \vert \lambda \vert $ for any infinite ordinal
$\lambda$. This should be later compared with Proposition
\ref{prop:siblingsurordinal} below, where we will see that
$\sib(\omega^{\alpha} \cdot \omega^*+ \omega^{\beta})=1$ if
$\alpha+1\leq \beta$, and $\sib(\sum^{*}_{n<\omega}\omega^{n}+ \gamma)=
\CC$ if $\gamma$ is any ordinal.

\begin{example}[Chain with many siblings]\label{exa:infinitesumord}
For any infinite ordinal $\lambda$, 
\[ \sib(\lambda^* \cdot \omega )= \vert \lambda \vert. \] 
\end{example}

\begin{proof}
Let $\lambda$ be an ordinal and $\vert \lambda \vert = \kappa$. Let
$\beta$ be the smallest ordinal such that $\lambda^* \cdot \omega
\equiv \beta ^* \cdot \omega$, and for every ordinal $\alpha < \beta$, let $C(\alpha) =
\alpha^* + (\beta^* \cdot \omega)$.  Then one readily sees that these chains $C(\alpha)$ are a
complete list of pairwise non-isomorphic siblings. Thus $\sib(\lambda^*
\cdot \omega)= \sib(\beta^* \cdot \omega)= \kappa = \vert \lambda
\vert$.
\qed \end{proof}

Hence  in particular there are chains with continuum many siblings. In
fact there are countable chains with that property, and here are two
examples that will be important for the sequel.

\begin{example}[Countable chain with continuum many siblings] \label{exa:ccsiblings}
\begin{enumerate}
\item $\sib(\BZ \cdot \omega)=\CC$. 
\item $\sib(\BQ)=\CC$.
\end{enumerate}
\end{example}

\begin{proof}
For the first part, let $C = \BZ \cdot \omega$. Now for $X \subseteq
\omega$ infinite, let $C(X) = \sum_{i \in \omega} \BZ^{\chi_i}$, where $\chi_i=1$ if $i \in X$, and
0 otherwise; this means we replace the copy of $\BZ$ by a singleton
for any index $i \notin X$.  Then clearly $C \equiv C(X)$ for any
infinite $X$, but $C(X) \not \cong C(Y)$ whenever $X \neq Y$, and thus
$\sib(C)=\CC$.

\medskip

But now one has that for any $ X \subseteq \omega$,
$\overline{C(X)}:=C(X)+\BQ \equiv \BQ$, and again $\overline{C(X)} \not
\cong \overline{C(Y)}$ whenever $X \neq Y$, and thus $\sib(\BQ)=\CC$. \qed
\end{proof}

This yields the following corollary.

\begin{corollary} \label{cor:smallsib}
$\sib(C)=\CC$ for all non-scattered countable chains. \\
\noindent Equivalently, if $C$ is a countable chain and $\sib(C)<\CC$, then $C$ is scattered. \\
\end{corollary}

Example \ref{exa:ccsiblings} also yields the
following.

\begin{corollary}\label{cor:variousstruc}
If a chain $C$ is an infinite alternating $\omega$-sequence of infinite
ordinals and reverse ordinals, then $\sib(C)\geq \CC$. \\ 

\noindent Similarly if a chain is of the form $C=\sum_{i \in \omega} \kappa_i^*$
(or its reverse) where the $\kappa_i$'s form a strictly  increasing chain of
cardinals (or even ordinals of strictly increasing cardinalities), then
$\sib(C)\geq max\{\CC, sup_i\{\kappa_i\}\}$.
\end{corollary}

We will show in Proposition \ref{prop:manysibling} that $\sib(C)\geq
\CC$ whenever $(\omega^*+\omega)\cdot \omega$ or $(\omega^*+\omega)
\cdot \omega^*$ are embeddable in a scattered chain $C$.

We hastily note that there are uncountable dense chains $C$ such that
$\sib(C)=1$, and one such construction is owed to Dushnik and Miller
\cite{key-DM} (see also Rosenstein \cite{key-R}).  Indeed they  have constructed, by a
clever $\CC$-length diagonalization, a dense uncountable subchain $C$
of the real numbers which is \emph{embedding rigid}, meaning that the
identity map is the only embedding of $C$ into itself. Clearly this
implies that $\sib(C)=1$.

We note that Baumgartner showed in \cite{key-B} that there are
$\kappa$-dense rigid chains of size $\kappa$ for each regular and uncountable
cardinal $\kappa$, although we do not know if these can be made
embedding rigid. Hence we ask the following. 

\begin{problem}\label{prob:kappaembeddingrigid}
Are there $\kappa$-dense embedding rigid chains of size $\kappa$ for
each regular and uncountable cardinal $\kappa$?
\end{problem}

\section{Structure Results}
\label{sec:struc}

In this section we describe two results characterizing the structure of a
chain and its number of siblings, which together yield the following dichotomy:

\begin{theorem} \label{thm:1orinfinite}
Let $C$ be any chain. Then $\sib(C)=1$ or $\sib(C) \geq \aleph_0$.
\end{theorem}

We will extend Corollary  \ref{cor:variousstruc} as follows:

\begin{proposition}\label{prop:manysibling}  
If $(\omega^*+\omega)\cdot\omega$ or $(\omega^*+\omega)\cdot \omega^*$ are
embeddable in a scattered chain $C$, then $\sib(C)\geq \CC$.
\end{proposition}

Chains in which neither $(\omega^*+\omega)\cdot\omega$ nor
$(\omega^*+\omega)\cdot\omega^*$ are embeddable have a special form
described in Proposition \ref{prop:finitesum} below. For this let us first 
recall the notion of ``surordinal'' introduced independently by Slater
\cite{key-S} and Jullien \cite{key-J1}.  A chain $C$ or its
order type is a \emph{surordinal} if for each $x\in C$ the cofinal
segment generated by $x$ is well ordered. Equivalently, $1+\omega^*$
does not embed into $C$. Jullien called \emph{pure} surordinals those
which are strictly left indecomposable (in particular, ordinals are
not pure). He observed that every non-pure surordinal is the sum of a
pure surordinal and an ordinal; the decomposition is not unique but
the pure surordinals figuring in two decompositions are equimorphic;
we will call \emph{component} such pure surordinal defined up to
equimorphy.  Jullien showed in his thesis that a surordinal is pure if
and only if it can be written as a sum $\sum^{*}_{n<\omega} C_n$
where each $C_n$ has order type $\omega^{\alpha_n}$ and the sequence
$(\alpha_{n})_{n<\omega}$ is non-decreasing. Furthermore, this sum is
unique (see Jullien \cite{key-J4} Proposition 3.3.2).

\begin{proposition}\label{prop:finitesum} 
Neither $(\omega^*+\omega)\cdot\omega$ nor $(\omega^*+\omega)\cdot\omega^*$
are embeddable into a chain $C$ if and only if $C$ is a finite sum of
surordinals and reverse of surordinals.
\end{proposition}

Thus, according to Propositions \ref{prop:manysibling} and
\ref{prop:finitesum}, scattered chains with few ($< \CC$) siblings are finite
sums of surordinals and their reverse.  In the next proposition, we
compute the number of siblings of a surordinal.

\begin{proposition}\label{prop:siblingsurordinal}
Let $C$ be a surordinal. Then:
\begin{enumerate}   
\item $\sib(C)=1$ if and only if either $C$ is an ordinal, 
$\omega^*$, or $C$ is not pure but the sequence in a component is
stationary, that is $C= \omega^{\alpha}\cdot \omega^*+
\omega^{\beta}+\gamma$ with $\alpha+1\leq \beta$ and $\gamma$ ordinal.
\item $\sib(C)= \vert C\vert $ if $C$ is pure and the sequence $(\alpha_{n})_{n<\omega}$ in the decomposition  of $C$ is stationary. 
\item $\sib(C)= \vert C'\vert ^{\aleph_0}$ if the sequence in a component $C'$ of $C$ is non-stationary.  
\end{enumerate}
\end{proposition}
 
The following result describe the scattered chains with few siblings. 
 
\begin{theorem}\label{thm:scattered} 
Let C be any chain and $\kappa < \CC$.  Then the following are equivalent:
\begin{enumerate}
\item $\sib(C)=\kappa$ and $C$ is scattered;
\item $\kappa=1$, or $\kappa\geq \aleph_0$ and $C$ is a finite sum 
of surordinals and of reverse of surordinals, and if $C=\sum_{j<m}
D_j$ is such a sum with $m$ minimum then $max\{sib (D_j): j<m\}=
\kappa$.
\end{enumerate}
\end{theorem}

This immediately yields the following interesting
corollary since, by Corollary \ref{cor:smallsib}, countable chains $C$
such that $\sib(C)<\CC$ must be scattered.

\begin{corollary} \label{cor:sib1}
When $C$ is countable, then $\sib(C)=1$,  $\aleph_0$, or  $\CC$.
\end{corollary} 

With an emphasis on the indecomposable components of $C$, we can also
prove a more general structure result when the number of siblings is
less than the continuum.

\begin{corollary}\label{cor:corollary2}
Let $C$ be a chain. Then:
\begin{enumerate} 
\item $C$ is scattered and $\sib(C)= \kappa < \CC$  
if and only if $C$ is a finite sum $\sum_{i<n} C_i$ of ordinals,
surordinals of the form $\omega^{\alpha}\cdot\omega^*+ \omega^{\beta}$
with $ \alpha+1 \leq \beta$, surordinals of the form
$\omega^{\alpha}\cdot\omega^*$ and reverse of such chains. \\
Furthermore if the number of components $C_i$ of this sum such that
$C_i$ or its reverse is of the form $\omega^{\alpha} \cdot \omega^*$ with
$\alpha\geq 1$ is minimum, then $\kappa$ is the maximum cardinality of
these components.
\item $\sib(C)$ is finite and $C$ is scattered if and only if  $C$ is a 
finite sum of ordinals, surordinals of the form
$\omega^{\alpha} \cdot \omega^*+ \omega^{\beta}$ with $ \alpha+1 \leq
\beta$, and their reverse. In which case, $\sib(C)=1$.
\end{enumerate} 
\end{corollary}

We can also prove a more general structure result when the number of
siblings is less than the continuum.

\begin{theorem}\label{thm:fullchar}
Let $C$ be any chain and $\kappa < \CC$. Then the following are equivalent:
\begin{enumerate}
\item $\sib(C)=\kappa $.
\item  $C = \sum_{i \in D} C_i$, where:
	\begin{itemize}
	\item $D$ is dense (singleton or infinite),
	\item each $C_i$ is scattered,
	\item $\sib(C_i)=1$ for all but finitely many $i \in D$, 
	\item $max\{\sib(C_i): i \in D\} = \kappa$, and 
	\item every embedding $f:C \rightarrow C$ preserves each $C_i$. 
	\end{itemize}
\end{enumerate}
\end{theorem}

Theorems \ref{thm:scattered} and \ref{thm:fullchar}
immediately prove Theorem \ref{thm:1orinfinite}. 

\begin{proof}[of Theorem \ref{thm:1orinfinite}]

If $C$ is  a chain such that  $\sib(C)=\kappa<\aleph_0$, then writing  $C=\sum_{i \in D} C_i$ as in Theorem
\ref{thm:fullchar}, we must have $\sib(C_i)=1$ for each $i \in D$ by Theorem \ref{thm:scattered} since each 
$C_i$ is scattered, and thus $\kappa = 1=\sib(C)$ from Theorem
\ref{thm:fullchar}. \qed
\end{proof}

\medskip

\noindent We also remark that the dense set $D$ in Theorem \ref{thm:fullchar}
does not have to be embedding rigid even when $\sib(C)=1$; in fact even $D =
\BR$ is possible in that case. Indeed Dushik and Miller \cite{key-DM} (see also
Rosenstein \cite{key-R}) showed that $\BR$ can be decomposed into two
disjoint dense subsets $E$ and $F$ such that $g(E) \cap F \neq
\emptyset$ and $g(F) \cap E \neq
\emptyset$ for any non-identity order preserving map $g:\BR
\rightarrow \BR$. Thus if $C=\sum_{i \in \BR} C_i$, where:
\[ \left\{ 
	\begin{array}{l} 
	|C_i|=2 \mbox{ if }  i \in E \\
	|C_i|=1 \mbox{ if }  i \not  \in E  \; (i \in F),\\
	\end{array}
	\right.  \]
then $C$ itself is embedding rigid. This is because given any order
preserving map $f:C \rightarrow C$, define a function $\phi: \BR
\rightarrow \powerset \BR $ such that $\phi(i)=\{j \in \BR: f(C_i) \cap C_j \neq
\emptyset \}$. Hence we can define $\overline{\phi(i)}$ as the
interval of $\BR$ determined by $\phi(i)$. Now we may define an order
preserving map $g:\BR \rightarrow \BR$ by:
\[ i \rightarrow \left\{ 
	\begin{array}{l} 
	j \mbox{ if }  f(C_i) \subseteq C_j \\
	\mbox{arbitrary } j \in E \cap \overline{\phi(i)} \mbox{ otherwise. }\\
	\end{array}
	\right. \]
But then $g(E)\subseteq E$ and hence $g$ is the identity map by
assumption, and this immediately implies that $f$ is the identity as well. 

\medskip 
On the other hand we can show that the dense set $D$ in Theorem \ref{thm:fullchar}
cannot be countably infinite.

\begin{proposition}\label{prop:allctble}
If $C = \sum_{i \in D} C_i$ where each $C_i$ is scattered and $D$ is a
countably infinite dense chain, then $\sib(C) \geq \CC$.
\end{proposition}

The proofs will be completed  in the next section, but we stress that there are
many immediate unanswered questions.

\begin{problem}\label{prob:embeddingrigid}
Suppose that $C = \sum_{i \in D} C_i$, where:
	\begin{itemize}
	\item $D$ is embedding rigid, 
	\item each $C_i$ is scattered,
	\item $\sib(C_i)=1$ for all but finitely many $i \in D$, and 
	\item $max\{\sib(C_i): i \in D\} = \kappa$.
	\end{itemize}
Does it follow that $\sib(C)=\kappa$?
\end{problem}

Another intriguing question is the following.

\begin{problem}\label{prob:embeddingrigiddecomp}
Suppose that a chain $C$ satisfies $\sib(C)=\kappa < \CC$, can $C$ be
in fact be written as in Problem \ref{prob:embeddingrigid}?
\end{problem}

And further regarding embedding rigidity, we cannot answer the
following question in full generality.

\begin{problem}\label{prob::embeddingrigiddouble}
Suppose that $C = \sum_{i \in D} C_i$, where $D$ and every $C_i$ are
embedding rigid, is $C$ necessarily embedding rigid?
\end{problem}

\noindent The answer here is clearly yes if $C$ is countable as this
immediately implies, since $D$ is embedding rigid, that $D$ is in fact
finite, and thus each $C_i$ must be finite as well. Similarly if each
$C_i$ is countable, then they must be finite. And again the answer is
positive if all the $C_i$ are isomorphic.

\section{Proofs} \label{sec:proof}

In this section we will prove Propositions \ref{prop:manysibling}, 
\ref{prop:finitesum}, \ref{prop:siblingsurordinal} and
\ref{prop:allctble}, and Theorems \ref{thm:scattered} and  \ref{thm:fullchar}.

Thus let $C$ be a chain. By Hausdorff's condensation arguments (see
Rosenstein \cite{key-R}) we can immediately write $C=\sum_{i \in D}
C_i$ where $D$ is dense (singleton or infinite) and each $C_i$ is
scattered. To see this, define, for $x,y \in C$, the equivalence
relation $x \equiv_0 y$ if the interval $[x,y]$ is finite. Now for
successor ordinals, define $x \equiv_{\alpha+1} y$ if the interval
$[x/{\equiv_\alpha},y/{\equiv_\alpha}]$ is finite in
$C/\equiv_{\alpha}$. For a limit ordinal $\beta$, simply let
$\equiv_\beta := \bigcup_{\alpha<\beta} \equiv_\alpha$. Then the
Hausdorff rank of $C$, written $h(C)$, is the least ordinal $\alpha$
such that $\equiv_\alpha \; = \; \equiv_{\alpha +1}$. Then $D$ above is $C /
\equiv_{h(C)}$ and the $C_i$ above are simply the $\equiv_{h(C)}$
equivalence classes.

\begin{proof} (of Propositions \ref{prop:finitesum})

Suppose that neither $(\omega^*+\omega) \cdot \omega$ nor
$(\omega^*+\omega) \cdot \omega^*$ are embeddable into $C$ and that every
chain $C'$ with this property and smaller Hausdorff rank is a finite
sum of surordinals and reverse surordinals.  Let $\alpha= h(C)$. If
$\alpha=0$ then $C$ is either an integer, $\omega$, $\omega^*$ or
$\omega^*+\omega$, hence an ordinal, the reverse of an ordinal or a
surordinal. Suppose $\alpha \geq 1$ and we proceed in two cases. 

\noindent  Case 1. $\alpha$ is a successor ordinal, $\alpha=\alpha'+1$.  
Then $D_{\alpha'}=C/\equiv_{\alpha'}$ is either an integer ($\not
=1$), $\omega$, $\omega^*$ or $\omega^*+\omega$. By the induction
hypothesis, each equivalence class $C_{n, \alpha'}$ of
$\equiv_{\alpha'}$ is a finite sum of surordinals and reverse
surordinals. Since $C= \sum_{n\in D_{\alpha'}} C_{n,
\alpha'}$, if $D_{\alpha'}$ is finite then $C$ is a finite sum of
surordinals and reverse surordinals too. If $D_{\alpha'}=\omega$ then
for $n$ large enough, either each $C_{n, \alpha'}$ is well ordered, or
reversely well ordered, otherwise $(\omega^*+\omega) \cdot \omega$ will be
embeddable into $C$, hence $C$ is a finite sum of surordinals and
reverse surordinals. The same argument leads to the same conclusion if
$D_{\alpha'}$ is equal to $\omega^*$ or to $\omega^*+\omega$. 

\noindent Case 2. $\alpha$ is a limit ordinal. 
In this case, every pair of elements $x$, $y$ of $C$ belongs to some
$\equiv_{\alpha'}$-equivalence class for some $\alpha'<\alpha$.  We
may write $C=A+B$ with $A= \sum^*_{\alpha<\kappa}A_{\alpha}$, $B=
\sum_{\beta<\lambda}B_{\beta}$, where $\kappa$ and $\lambda$ are
cardinals equal respectively to the coinitiality and the cofinality of
$C$ and the $A_{\alpha}$'s and $B_\beta$'s included in some
$\equiv_{\alpha'}$-equivalence classes. If $\lambda\geq\omega$ then
for $\beta$ large enough, say $\beta\geq \beta_0$, $B_{\beta}$ is well
ordered, or reversely well ordered (otherwise again 
$(\omega^*+\omega) \cdot \omega$ would be embeddable into $C$). Since by
induction each $B_{\beta}$ is a finite sum of surordinals and reverse
surordinals, $\sum _{\beta\geq \beta_0}B_{\beta}$ is such a sum as well.
Since via the induction hypothesis $\sum_{\beta'<\beta_0}B_{\beta'}$
is a finite sum of surordinals and reverse surordinals, $B$ is thus such a
sum.  The same argument applied to $A$ ensures that $C$ is a sum
of surordinals and reverse of surordinals.  

Conversely, suppose that $C$ is a finite sum of surordinals and
reverse of surordinals.  If $(\omega^*+\omega) \cdot \omega$ was embeddable
into $C$, then since it is indecomposable, it would be embeddable into a
member of this sum, which is clearly impossible.  The same argument applies to
$(\omega^*+\omega) \cdot \omega^*$. With that the proof is complete. \qed
\end{proof}

But the goal here is to have a more specific structure
decomposition. To do so we shall make use of labellings of a chain $C$
by a well quasi-ordered set $Q$, that is a reflexive and transitive
binary relation such that every infinite sequence contains an infinite
increasing subsequence. A labelling of $C$ is a pair of the form
$(C,\ell)$ (or simply $\ell$ when $C$ is clear) where $\ell:C \rightarrow Q$ is an order preserving map. A
$Q$ embedding of a labelling $(C,\ell)$ into another labelling
$(C',\ell')$ is an embedding $f:C \rightarrow C'$ such that $\ell(x)
\leq \ell'(f(x))$ for all $x \in C$.
Two labellings $\ell$ and $\ell'$ (or $(C,\ell)$ and $(C,\ell')$) are
said to be isomorphic if there is an automorphism $\phi$ of $C$ such
that $\ell' \circ \phi = \ell$.

We begin with a counting argument for labellings. 

\begin{lemma}\label{lem:labelling}
Let $C$ be an infinite chain and $|Q|>1$. Then there are at least $\CC$ 
pairwise non-isomorphic labellings.
\end{lemma}

\begin{proof}
Observe that if $|C|=\mu\geq \aleph_0$, $|Q|=\kappa>1$, and $|Aut(C)|<
\kappa^\mu$, then there are $\kappa^\mu \geq \CC$ non-isomorphic
labellings.

In general write $C= \sum_{i \in C/ \equiv_0} F_i$, and thus each $F_i$ is either
finite, or type $\omega$, $\omega^*$, or $\omega^* + \omega$. We
proceed in cases.

We first consider the case where some $F_i$ is infinite. Then
selecting an arbitrary $q \in Q$, we extend each labelling $\ell:F_i
\rightarrow Q$ to the labelling $\overline{\ell}:C \rightarrow Q$ by
setting $\overline{\ell}(x)=\ell(x)$ for $x \in F_i$, and
$\overline{\ell}(x)=q$ otherwise. If now two such labellings
$\overline{\ell}$ and $\overline{\ell}'$ are isomorphic via some
automorphism $\phi$ of $C$, for example $\overline{\ell}' \circ \phi =
\overline{\ell}$, then $\phi(F_i)=F_j$ for some $j$. If $j \neq i$,
then $\ell'$ and $\ell$ are equal to $q$ on $F_i$; if $j=i$, then
$\phi$ induces an automorphism of $F_i$ and $\ell'$ and $\ell$ are
isomorphic. Hence there are as many isomorphic labellings of $C$ as of
$F_i$, and thus at least $\CC$.

In the case all the $F_i$ are finite, then $C/ \equiv_0$ is infinite
and dense, and thus $C$ contains a copy of every countable chain. But
each such copy yields a labelling of $C$ into two colours, and since
isomorphic labellings yield isomorphic copies and there are $\CC$
non-isomorphic countable chains, then there are $\CC$ non-isomorphic
labellings of $C$.\qed
\end{proof}

In particular, if $C$ is equal to $\omega$, $\omega ^*$ or to
$\omega^*+\omega$, then  there are $|Q|^{\aleph_0}$ non-isomorphic
labellings. 

\begin{problem}\label{prob:labellings}
Are there generally in fact $2^{|C|}$ non-isomorphic labellings? 
\end{problem}

\medskip

We will also need the following.

\begin{lemma} \label{lem:sib}
Let $C$ be a chain, $\alpha$ an ordinal and $\kappa$ a cardinal. \\
\noindent Then if $\CM=\{E \in C / \equiv_\alpha : \sib(E)\geq \kappa\}$ and
$\mu=|\CM|$, then $\sib(C) \geq min\{ \CC, \kappa^\mu\}$.
\end{lemma}

\begin{proof}
For $E \in \CM$, let $S(E)$ be a collection of $\kappa$ pairwise
non-isomorphic chains equimorphic to $E$. Then for $\zeta \in \prod_{E
\in
\CM} S(E)$, define 
\[ C(\zeta) = \sum_{E\in C/ \equiv_\alpha} C_{E, \zeta} \]
where 
\[ C_{E,\zeta}= \left\{ 
	\begin{array}{l} 
	 \zeta(E) \mbox{ if } E \in \CM, \\
	  E \mbox{ otherwise.} \\
	\end{array}
	\right.
\]
Clearly $C(\zeta) \equiv C$ for each $\zeta$. Now for $\zeta, \xi \in
\prod_{E \in \CM} S(E)$, an isomorphism between $C(\zeta)$ and $C(\xi)$ 
would  preserve $\equiv_\alpha$ classes, and thus  would induce an isomorphism
$g$ of $C / \equiv_\alpha$ onto $ C / \equiv_\alpha$ such that $E
\equiv g(E)$ for each $E \in C / \equiv_\alpha$. Clearly $E \in \CM$
if and only if $g(E) \in \CM$, and thus $\zeta(E)=\xi(g(E))$. This
means that the labelled chains $(\CM,\zeta)$ and $(\CM, \xi)$ are
isomorphic.

\noindent In the case that $\CM$ is finite then we must have 
$\zeta=\xi$, and the number of non-isomorphic labellings is at least
\[  \prod_{E \in \CM} |S(E)| = max \{ |S(E)|: E \in \CM \} \geq \kappa. \]
If on the other hand $\CM$ is infinite, then according to Lemma
\ref{lem:labelling} there are at least $\CC$ non-isomorphic labellings
into a set of size at least 2. The conclusion follows.  \qed
\end{proof}

\begin{lemma}\label{lem:interval}
If $f:C \rightarrow C$ is an order preserving map, and for some $x \in
C$ the interval determined by $x$ and $f(x)$ is non-scattered, then $\sib(C)\geq \CC$.
\end{lemma}

\begin{proof}
We may assume without loss of generality that $x<f(x)$, and define 
\[ A=(-\infty,x], \;  M=(x,f(x)), \; \mbox{ and } B=[f(x),+\infty) \]
Thus $f$ is  a witness to $C \leq A+B$. But then $A+X+B \equiv C$ whenever $X \leq M$, since for
any such $X$ we have:
\[ A+X+B \leq C \leq A+B \leq A+X+B.\]
But now, if $\{S_\alpha: \alpha < \CC\}$ is a family of pairwise
non-isomorphic countable scattered sets such that $\sib(S_\alpha)>1$,
then let $X_\alpha =
\BQ+S_\alpha+\BQ$ and finally $C_\alpha = A + X_\alpha+B$. From the
above remark we immediately have that $C_\alpha \equiv C$ for each
$\alpha$.

Hence $\sib(C) \geq \CC$ provided  that the $C_\alpha$'s are
pairwise non-isomorphic. Suppose on the contrary that $g:C_\alpha
\rightarrow C_\beta$ is an isomorphism for some $\alpha \neq \beta$. 

\noindent Write $C_\alpha = \sum_{i \in D_\alpha} C_{\alpha, i}$ and $C_\beta =
\sum_{i \in D_\beta} C_{\beta, i}$, where each $C_{\alpha, i}$ and $
C_{\beta, i}$ is scattered and $D_\alpha$ and $D_\beta$ are dense.

We claim that $\{i\in D_{\alpha}: \sib(C_{\alpha,i})>1\}$ is
infinite. Indeed, $g$ carries each $C_{\alpha, i}$ to some $C_{\beta,
j}$. 

Since $S_{\alpha}$ is some $C_{\alpha,i}$, $S_{\beta}$ some
$C_{\beta,j}$ and $S_{\alpha}\not \simeq S_{\beta}$, the image of
$S_{\alpha}$ is either included into $A$ or into $B$. Without loss of
generality, we can assume that this image is included into $B$; from
which it follows that $g$ is an embedding of $B$ into itself. Consider
an arbitrary element $x_0 \in S_{\alpha}$, recursively define
$x_{n+1}=g(x_{n})$ for $n<\omega$.  Without loss of generality we may
assume that $x_0<x_1=g(x_0)$, and thus observe that $g$ is an
embedding of $B$ into itself.  Now for each $n$ choose $i_n$ such that
$x_n \in C_{\alpha, i_n}$, but then $C_{\alpha, i_n} \cong S_\alpha$,
and thus $\sib(C_{\alpha, i_n})\geq 2$.  Applying Lemma \ref{lem:sib}
with $\alpha=h(C)$, we get $\sib(C_\alpha) \geq \CC$, and hence $\sib(C)
\geq \CC$. \qed
\end{proof}

We will need two other notions of condensation. Let $C$ be a chain,
define for $x,y\in C$ the equivalence relation $x\equiv_{well} y$
(resp. $x\equiv_{well^*} y$) if the interval $[x,y]$ is well ordered
(resp. reversely well ordered), see Rosenstein \cite {key-R} pages 72
and 73 for an illustration of these notions.

We need an easy observation whose verification is left to the reader.

\begin{lemma}\label{lem:equivclasses}
Equivalence classes of $\equiv_{well}$ (resp. $\equiv_{well^*}$) are
the maximal intervals of $C$ which are surordinals (resp. reverse
surordinals). \\
\noindent Furthermore, the intersection of $\equiv_{well}$ and
$\equiv_{well^*}$ is equal to $\equiv_0$.
\end{lemma}

\begin{lemma}\label{lem:meet}
Let $A$, $B$ two intervals of a chain $C$, one being a
$\equiv_{well}$-equivalence class and the other an
$\equiv_{well^*}$-class. If  $A\cap B$, $A\setminus B$ and
$B\setminus A$ are each non-empty, then $A\cap B$ is either finite or has
order type $\omega^*+\omega$.
\end{lemma}

\begin{proof}
From Lemma \ref{lem:equivclasses} above, $A\cap B$ is a
$\equiv_0$-class. Hence it is either finite or has type $\omega^*$,
$\omega$ or $\omega^*+\omega$. We claim that the types $\omega^*$ and
$\omega$ do not occur. Indeed, suppose that $A\cap B$ has type
$\omega$ (the case $\omega^*$ is similar). Let $a\in A\setminus B$ and
$b\in B\setminus A$. Without loss of generality we may suppose $a<b$
(otherwise, exchange the names of $A$ and $B$). Let $c$ be the least
element of $A\cap B$. Since the interval $[a, b]$ contains a chain of
type $\omega$ it is not dually well ordered, hence $B$ is a
$\equiv_{well}$-equivalence class (and since it has a least element it
forms a well ordered chain). Now, $A$ must be a
$\equiv_{well^*}$-equivalence class, but this is impossible. Indeed,
otherwise $[a, c]$ is dually well ordered, hence it contains a lower
cover $c'$ of $c$; but this lower cover satisfies $c'\equiv_0 c$ and
$c'\not\in A\cap B$, contradicting the fact that $A\cap B$ is a
$\equiv_0$-class. \qed 
\end{proof}

\begin{lemma}\label{lem:numbersibling}
For each $\equiv_{well}$-equivalence class or
$\equiv_{well^*}$-equivalence class $E$ of a chain $C$, we have
$\sib(C)\geq \sib(E)$.
\end{lemma}

The proof follows the same lines of the proof of Lemma \ref{lem:sib}.

\medskip

We recall the notions of indecomposability (from Rosenstein
\cite{key-R} and Fra\"{\i}ss\'e \cite{key-F}). A chain $C$ is
(additively) \emph{indecomposable} if for every decomposition of $C$
into an initial segment $A$ and a final segment $B$, $C$ is embeddable
either into $A$ or into $B$; it is \emph{left indecomposable} if it is
embeddable into every non-empty initial segment and it is
\emph{strictly left indecomposable} if for every decomposition into a
non-empty initial interval $A$ and a final interval $B$, $C$ is
embeddable into $A$ but not into $B$. Right and strictly right
indecomposability are defined in the same way. We recall that if $C$
is indecomposable (resp. strictly right or left indecomposable) and
$C'\equiv C$ then $C'$ is indecomposable (resp. strictly right and left
indecomposable).  We also recall that indecomposable ordinals coincide
with ordinals of the form $\omega^{\alpha}$; that scattered
indecomposable chains are either strictly left indecomposable or
strictly right indecomposable and that every scattered chain is a
finite sum of indecomposable chains, a quite non-trivial result of
Laver \cite{key-L1}.  

We will need the following result of Jullien \cite{key-J3}.

\begin{lemma} \label{finalsegment}
Let $C$ be a chain. Suppose that $C$ is embeddable in each non-empty
final segment of $C$. If $C$ is not an ordinal then $\sib(C)\geq
\aleph_0$.
\end{lemma}

\begin{proof} 
First for an arbitrary chain $D$, let $I_D$ be the largest well ordered initial
segment of $D$ and let $D':= D \setminus I_D$. 

\noindent Assuming that the chain $C$ is not an ordinal, and thus $C'$ is non-empty, 
then we claim that $C_n=n+C'$ is equimorphic to $C$ for each $n \in
\omega$, and that the $C_n$'s are pairwise non-isomorphic, thus
proving our claim.

\noindent First the $C_n$ are not isomorphic since $I_{C_{n}}=n$. Further, from
the fact that $C$ is embeddable in each non-empty final segment, then
$C$ immediately embeds in $n+C'$, and moreover since $C'$ is infinite
then $C$ (and thus $n+C'$) is embeddable in $C$ for each $n$. \qed
\end{proof}

We now describe siblings of a finite sum of indecomposable order types. 

\begin{lemma}\label{indec1} 
Let $\alpha$ be an order type which is a finite sum of indecomposable
order types. If $\alpha'\equiv\alpha$, then $\alpha'$ is a finite sum
of indecomposable order types.  \\
\noindent Let $n$, resp $n'$, be minimal such
that $\alpha = \alpha_0+ \cdots+ \alpha_{n-1}$ (resp.  $\alpha' =
\alpha'_0+ \cdots+ \alpha_{n'-1}$) where each $\alpha_i$
(resp. $\alpha'_{i'}$) is indecomposable. Then $n'=n$ and
$\alpha'_i\equiv \alpha_i$ for $i<n$.
\end{lemma}

For reader's convenience we give a proof. We record the proof given in
Pouzet \cite{key-P} (see Proposition II-5.7).

\begin{proof}  
We say that a decomposition of $\alpha= \alpha_0+ \cdots+
\alpha_{n-1}$, where each $\alpha_i$ is indecomposable, is
\emph{minimal} if $\alpha_i+\alpha_{i+1}$ is not indecomposable for
$i+1<n$. Clearly, a decomposition of minimal length is minimal (our
proof will show in particular that the converse holds). 

Let $\alpha= \alpha_0+ \cdots+ \alpha_{n-1}$ be a minimal
decomposition and $\alpha'\equiv \alpha$. Let $A'\subseteq \alpha$
with $A'$ of type $\alpha'$. Set $A'_i= A'\cap
\alpha_i$ for $i<n$. Clearly, $A'= \sum_{i<n} A'_i$. We claim that
$A'_i\equiv\alpha_i$ for every $i<n$, proving that the $A'_i$'s form a 
decomposition of $A'$. Indeed, since $\alpha\leq \alpha'$ there is an
embedding $f$ of $\alpha$ into $A'$. Let $i<n$; since $\alpha_i$ is
indecomposable, there is some $k<n$ such that $\alpha_i \equiv
\alpha_i \cap f^{-1}(A'_k)$; let $\varphi(i)$ be the least $k$ with
this property. This allows to define a map $\varphi$ from $n$ into
$n$. This map being order preserving then, as observed by Jullien
\cite{key-J4} Lemma 3.4.1, if it is not the identity there
is some $i<n$ such that either $i=\varphi(i)=\varphi(i+1)$ or
$i=\varphi(i-1)=\varphi(i )$ (the first case occurs if there is
$x<f(x)$ and the second case if there is some $x>f(x)$). 

In the first case $\alpha_i+ \alpha_{i+1}\leq A'_i\leq \alpha_i$,
proving that $\alpha_i+ \alpha_{i+1}$ is indecomposable and
contradicting the minimality of the decomposition of $\alpha$; the
second case is similar. 

Thus $\varphi$ is the identity map. Consequently $A'_i\equiv \alpha_i$
for all $i$ as claimed. We may observe furthermore that the
decomposition of $A'$ induced by the decomposition of $\alpha$ is
minimal.  Let $\alpha' = \alpha'_0+ \dots \alpha'_{n'-1}$ be a minimal
decomposition. Without loss of generality, we may suppose $n'\geq
n$. This decomposition of $\alpha'$ induces a decomposition of $A'$ as
$A'= \sum_{i'<n'} A''_{i'}$. As above, we may define $\varphi'$ from
$n'$ into $n$, setting $\varphi(i')=k$ where $k$ is minimum such that
$A''_{i'}\equiv A'_k$ and $A'_k\cap A''_{i'}\not = \emptyset$. Due to
the minimality of the decomposition of $\alpha'$, $\varphi'$ is one to
one, hence the identity, hence $n'=n$ and $\alpha'_i\equiv \alpha_i$
for $i<n$ as required. \qed \end{proof}

We are now ready to start counting the number of siblings of finite
sum of indecomposable order types.

\begin{lemma}\label{lem:indec2} 
Let $\alpha$ be an order type which is a finite sum of indecomposable
order types and let $\alpha = \alpha_0+ \cdots+ \alpha_{n-1}$ be a
minimal decomposition. \\
\noindent Then  for each $i<n-1$,   $\sib(\alpha)\leq
\sib(\alpha_0+\cdots +\alpha_i)\times \sib(\alpha_{i+1}+ \cdots+
\alpha_{n-1})$. \\
\noindent Furthermore, equality holds if $\alpha
_i$ is strictly right indecomposable or $\alpha_{i+1}$ is strictly
left indecomposable.
\end{lemma}

\begin{proof}
Set $\overline \alpha_i:= \alpha_0+ \dots \alpha_{i}$ and $\underline
\alpha_{i+1}:= \alpha_{i+1}+ \dots \alpha_{n-1}$.  Let $(\overline
\beta'_i, \underline \beta'_{i+1})\in \sib(\overline \alpha_i)\times
\sib(\underline \alpha_{i+1})$ and $\vartheta (\overline \beta'_i,
\underline \beta'_{i+1})= \overline \beta'_i+ \underline
\beta'_{i+1}$. Clearly, $\vartheta$ maps $\sib(\overline
\alpha_i)\times \sib(\underline \alpha_{i+1})$ into $\sib(\alpha)$. We
claim that this maps in surjective.  Indeed, let $\alpha'\equiv
\alpha$, then according to Lemma \ref{indec1},  $\alpha' = \alpha'_0+
\dots \alpha_{n-1}$ with $\alpha'_i\equiv \alpha_i$ for $i<n$.
Setting $\overline \alpha'_i:= \alpha'_0+ \cdots+ \alpha'_{i}$ and
$\underline \alpha'_{i+1}:= \alpha'_{i+1}+ \cdots+ \alpha'_{n-1}$, we
have $(\overline \alpha'_i, \underline \alpha'_{i+1})\in \sib(\overline
\alpha_i)\times \sib(\underline \alpha_{i+1})$ and $\vartheta(\overline
\alpha'_i, \underline \alpha'_{i+1})=\alpha'$. This proves the
surjectivity; the inequality between cardinals follows. 

Now we prove that this map is one to one provided that $\alpha _i$ is
strictly right indecomposable or $\alpha_{i+1}$ is strictly left
indecomposable. We suppose that $\alpha _i$ is strictly right
indecomposable (the case $\alpha_{i+1}$ strictly left indecomposable
is similar). We prove that if $\alpha'\in \sib(\alpha)$ there is a
unique pair $(\overline \alpha'_i, \underline \alpha'_{i+1})\in
\sib(\overline \alpha_i)\times \sib(\underline \alpha_{i+1})$ such that
$\vartheta(\overline \alpha'_i, \underline
\alpha'_{i+1})=\alpha'$. The existence of such a pair was proved
above. If $(\overline \alpha''_i, \underline \alpha''_{i+1})$ is an
other pair, then either $\overline \alpha''_i$ is an initial segment
of $\overline \alpha'_i$ or $\overline \alpha'_i$ is an initial
segment of $\overline \alpha''_i$. We may suppose that $\overline
\alpha''_i$ is an initial segment of $\overline \alpha'_i$. We prove
that if fact $\overline \alpha''_i=\overline \alpha'_i$ and thus
$\underline \alpha''_{i+1}= \underline \alpha'_{i+1}$, yielding the
injectivity as required. Indeed, since $\alpha_i$ is strictly right
indecomposable and $\alpha'_i\equiv \alpha_i$ then $\alpha'_i$ is
strictly right indecomposable. If $\overline \alpha''_i$ is a proper
initial segment of $\overline \alpha'_i$, let $u$ such that $\overline
\alpha'_i= \overline \alpha''_i+u$. Then $u$ is a proper final segment
of $\alpha'_i$ (otherwise $\overline \alpha''_i\leq \alpha'_0+ \cdots+
\alpha'_{i-1}<\overline \alpha''_i$ which is impossible since
$\overline \alpha''_i\equiv \overline \alpha''_i$).  Let $v$ such that
$\alpha'_i=v+u$; then $\overline \alpha''_i=\alpha''_0+ \cdots+
\alpha''_{i-1}+v$. Since this is a minimal decomposition of $\overline
\alpha_{i}$ we have $v\in \sib(\alpha_i)$ and since $\alpha'_i$ is
strictly right indecomposable, we have $v<\alpha'_i$, a contradiction. \qed
\end{proof}

From this we deduce the following two extensions we will need. 

\begin{corollary} \label {cor:indec} 
Let $\alpha$ be an order type which is a finite sum of indecomposable
order types and let $\alpha = \alpha_0+ \cdots+ \alpha_{n-1}$ be a
minimal decomposition. \\
If $\alpha _i$ is strictly left indecomposable and $ \alpha_{i+1}$ is
strictly right indecomposable, then 
\[ \sib(\alpha)=\sib(\alpha_0+\cdots+
\alpha_{i-1})\times \sib(\alpha_i+\alpha_{i+1})\times \sib(\alpha_{i+2}+
\cdots +\alpha_{n-1}).\]
\end{corollary}

\begin{corollary}\label{cor:comput}
Let $\alpha$ be an order type which is a finite sum of indecomposable
order types and let $\alpha = \alpha_0+ \cdots+ \alpha_{n-1}$ be a
minimal decomposition where each $\alpha_i$ is either strictly left or strictly right
indecomposable. Now let 
\[K =\{i<n-1: 
	\begin{array}{l} \alpha _i \mbox{ is  strictly left indecomposable and } \\
	 \alpha_{i+1}  \mbox{ is  strictly right indecomposable.} \} 
	\end{array} \]
Then,
\begin{equation} \label{formula1}
\sib(C)= \prod_{i\not \in K} \sib(\alpha_i)\times \prod_{i\in K} \sib(\alpha_i+\alpha_{i+1}).
\end{equation}
\end{corollary}

Now let $\kappa$ be a cardinal. Following Laver in \cite {key-L1}, 
we call a family $(\alpha_\lambda)_{\lambda<
\kappa}$ of order types \emph{unbounded} if for every
$\lambda<\kappa$ the set of $\mu$ such that $\alpha_{\lambda}\leq
\alpha_{\mu}$ has cardinality $\kappa$.  According to
Laver \cite{key-L1}, if $\kappa$ is regular and
$(\alpha_\lambda)_{\lambda< \kappa}$ is unbounded then $\sum_{\lambda
<\kappa}\alpha_{\lambda}$ and $\sum^*_{\lambda
<\kappa}\alpha_{\lambda}$ are indecomposable, respectively on the
right and on the left.

The following result is essentially in Jullien \cite{key-J4} and also
Laver
\cite{key-L1}.

\begin{lemma} \label{lem:wqo}
Let $\mathcal C$ be a collection of order types which is closed
downward under embeddability, that is $\alpha \in \mathcal C$ and
$\beta\leq \alpha$ imply $\beta\in \mathcal C$. Then
\begin{enumerate}
\item $\mathcal C$ is well quasi ordered under embeddability if and only 
if the collection $Ind(\mathcal C)$ of indecomposable members of
$\mathcal C$ is well quasi ordered and every member of $\mathcal C$ if
a finite sum of members of $Ind(\mathcal C)$.
\item If $\mathcal C$ is well quasi ordered under embeddability, then every unbounded $\kappa$ 
sequence $(\alpha_\lambda)_{\lambda< \kappa}$ with $\kappa$ regular
has a final sequence which is unbounded.
\item Let $\alpha$ such that for every $x<y \in \alpha$ the interval $[x,y]$ 
belongs to $\mathcal C$. If $\mathcal C$ is well quasi ordered, then
$\alpha$ is a finite sum of indecomposable order types $\alpha_0+
\dots + \alpha_{n-1}$ where:
	\begin{itemize}
	\item  $\alpha_{i}\in Ind(\mathcal C)$ for $i\not
	= 0, n-1$, 
	\item $\alpha_0\in Ind(\mathcal C)$ or $\alpha_{0}$ is a reverse
	ordinal sum of a regular unbounded sequence of members of
	$Ind(\mathcal C)$, 
	\item $\alpha_{n-1}\in Ind(\mathcal C)$ or $\alpha_{n-1}$
	is an ordinal sum of a regular unbounded sequence of members of
	$Ind(\mathcal C)$.
	\end{itemize}
\end{enumerate}
\end{lemma}

The proof of (1) relies on Higman's theorem on words
\cite {key-H}. For the proof of (2), set $I(F):= \{\beta\in \mathcal
C: \beta \leq \alpha_{\lambda}  \mbox{ for some } \lambda\in F\}$
for each final segment $F$ of $\kappa$. Since $\mathcal C$ is well
quasi ordered, the set of its initial segments is well founded (Higman
\cite{key-H}). Let $F_0$ such that $I(F)$ is minimal. Then
$(\alpha_\lambda)_{\lambda\in F_{0}}$ is unbounded. The proof of (3)
follows from (1) and (2).

\medskip

We will use the following consequence of Lemma \ref{lem:wqo}. Let $S$
be the class of surordinals. As proved directly by Jullien \cite
{key-J1}, $S$ is well quasi ordered, hence the class $U$ of
surordinals and their dual is well quasi ordered. Thus (by (1) of
Lemma \ref{lem:wqo}), the class $\Sigma (U)$ of finite sums of members
of $U$ is well quasi ordered. According to (2) of Lemma \ref{lem:wqo},
if each interval $[x,y]$ of a chain $C$ belongs to $\Sigma (U)$, then
$C$ is a finite sum $C_0+ \dots + C_{n-1}$ where $C_{i}\in Ind(U)$ for
$i\not = 0, n-1$, $C_0$ and $C_{n-1}$ are a reverse ordinal sum and an
ordinal sum of regular unbounded sequences of members of $Ind(U)$.

A famous theorem of Laver \cite{key-L1}, answering positively
Fra\"{\i}ss\'e's conjectures on chains asserts that the class
$\mathcal D$ of scattered order types is well quasi ordered under
embeddability (see the exposition by Rosenstein in \cite{key-R} and
Fra\"{\i}ss\'e in \cite{key-F}). From this follows that every
scattered order type is a finite sum of indecomposable order types. A
consequence is the following Lemma \ref{lem-0-equiv finie}.

\begin{lemma} \label{lem-0-equiv finie} 
If a scattered chain $C$ is indecomposable and infinite, there is some
equimorphic chain $C'$ such that all the $\equiv_0$-equivalence
classes are infinite.
\end{lemma}

We need a weaker statement, namely the conclusion of this lemma when
$C$ is an infinite member of $Ind(U)$ or an $\omega$-sum of an
unbounded sequence of members of $Ind(U)$.  This does not require the
well quasi order of $\mathcal D$, which is proved by means of
Nash-Williams theory of better quasi ordering, but only the well quasi
ordering of $Ind(U)$. The interest of this weakening could be in the
programme of reverse mathematics, see for example Montalb\'an
\cite{key-M}. The proof of this weakening is straightforward. Given a
chain $C$, let $F_{\equiv_0} (C)$ be the set of $x\in C$ such that the
$\equiv_0$-class of $x$ is finite. Note that $F_{\equiv_0}(C)$ is
empty if $C$ is a surordinal without a largest element, hence
$F_{\equiv_0}(C)$ is empty if $C$ is an infinite member of
$Ind(U)$. Now if $C$ is an $\omega$-sum of an unbounded sequence of
members of $Ind(U)$, say $C=\sum_{n<\omega} C_n$, let $A$ be the set
of $C_n$ which are infinite (and in fact have more than one
element). If $A$ is finite then it is empty and $C$ is the chain
$\omega$, and  hence $F_{\equiv_0}(C)$ is empty. If $A$ is infinite, let
$C'= \sum_{n\in A} C_n$.  Then $F_{\equiv_0}(C')$ is empty; now, as it
is easy to see using the unboundedness of the sequence of $C_n$, $C'$
is equimorphic to $C$, giving the required conclusion.

%
%
%
%
%

%

\medskip

We now tackle the proofs of Propositions \ref{prop:siblingsurordinal}
and \ref{prop:manysibling} which will lead us to that of Theorem
\ref{thm:scattered} on the characterization of scattered sets with a
small number of twins.

\begin{proof} (of  Proposition  \ref{prop:siblingsurordinal})

We may suppose that $C$ is not an ordinal (otherwise $\sib(C)=1$). In
this case, $C$ decomposes as $C'+D$ where $C'= \sum^*_{n<\omega}C_n$,
$(C_n)_{n<\omega}$ is a non-decreasing sequence of ordinals of type
$\omega^{\alpha_n}$ and $D$ is an ordinal, this ordinal being $0$ if
$C$ is pure and, otherwise, a non-increasing sequence $D_0+\cdots
+D_{n_0}$ of ordinals $D_i$ of type $\omega^{\beta_i}$ with
$\alpha_n+1\leq\beta_0$ for every $n$.

If $(C_n)_{n<\omega}$ is stationary and $\alpha=
max\{\alpha_n:n<\omega\}$, then we may rewrite $C$ as $\omega^{\alpha}
\cdot \omega^*+\gamma$, with $\gamma<\omega^{\alpha+1}$ if $C$ is
pure, and $\omega^{\alpha} \cdot \omega^*+\omega^{\beta}+\gamma$ where
$\alpha+1\leq \beta$ otherwise. This yields $\sib(C)=\vert
\omega^{\alpha+1}\vert=\vert C\vert$ if $C$ is pure and $\alpha\not =
0$, and $\sib(C)=1$ otherwise. 

If $(C_n)_{n<\omega}$ is non-stationary, then $\sib(C)=\vert C'\vert
^{\aleph_0}$. Indeed, let $\mu=sup\{\omega^{\alpha_{n}}+1:n<\omega\}$
and $(C'_n)_{n<\omega}$ be a non-decreasing sequence of ordinals of
types $\omega^{\alpha'_{n}}$ such that
$sup\{\omega^{\alpha'_{n}}+1:n<\omega\}= \mu$. Then
$\sum^*_{n<\omega}C'_n$ is equimorphic to $\sum^*_{n<\omega}C_n$ and
is isomorphic to this sum if and only if $(C'_n)_{n<\omega}=
(C_n)_{n<\omega}$. Since there are at least $\vert \mu\vert
^{\aleph_0}$ such sequences and $\vert \mu\vert ^{\aleph_0}= \vert
C'\vert ^{\aleph_0}$, this does the case $D=0$. If $D\not =0$, the
fact that $\sum^*_{n<\omega}C'_n+ D$ is isomorphic to $C$ implies that
some final sequence of $(C'_n)_{n<\omega}$ coincide with a final
sequence of $(C_n)_{n<\omega}$. This again yields at least $\vert
\mu\vert ^{\aleph_0}= \vert C'\vert ^{\aleph_0}$ equimorphic
types. 

This completes the proof of the proposition. \qed
\end{proof}

\begin{proof} (of Proposition \ref{prop:manysibling})

We prove this proposition by induction on the Hausdorff rank of the
given chain $C$.  Suppose that either $(\omega^*+\omega) \cdot \omega$
or $(\omega^*+\omega) \cdot \omega^*$ are embeddable into $C$ and that
$\sib(C')\geq \CC$ for every chain $C'$ with this property and smaller
Hausdorff rank.  Let $\alpha= h(C)$, and note that necessarily $\alpha
\geq 1$. Now for $\alpha'<\alpha$, if $C'$ is an
$\equiv_{\alpha'}$-equivalence class, Lemma \ref{lem:sib} ensures that
$\sib(C)\geq \sib(C')$, hence if $\sib(C')\geq \CC$ we will have
$\sib(C)\geq \CC $ as required. This will be the case if either
$(\omega^*+\omega) \cdot \omega$ or $(\omega^*+\omega) \cdot \omega^*$
are embeddable into $C'$.  Indeed, the induction hypothesis insures
that $\sib(C')\geq \CC$. Hence we may suppose that neither
$(\omega^*+\omega) \cdot \omega$ nor $(\omega^*+\omega) \cdot
\omega^*$ are embeddable into $C'$ and $\sib(C')< \CC$ (in fact, with
Lemma \ref{lem:sib} we can suppose that except for finitely many, all
$\equiv_{\alpha'}$-equivalence classes $C'$ satisfy $\sib(C')=1$).

\noindent We claim that we may write $C$ as a finite sum $C=C_0+
\cdots +C_{n-1}$ where $C_{i}\in Ind(U)$ for $i\not = 0, n-1$,
$C_{0}\in Ind(U)$ or $C_{0}$ an unbounded $\omega^*$-sum of members
of $U$, namely $C_{0}= \sum^*_{n<\omega}A_{n}$, $C_{n-1}\in Ind(U)$ or
$C_{n-1}$ is an unbounded $\omega$-sum $C_{n-1}=\sum_{m<\omega}
B_{m}$.  For that claim, we distinguish two cases. \\
\noindent Case 1: $\alpha$ is
a successor ordinal, $\alpha=\alpha'+1$. In this case,
$D_{\alpha'}=C/\equiv_{\alpha'}$ is either an integer ($\not =1$),
$\omega$, $\omega^*$ or $\omega^*+\omega$. The chain $D$ cannot be
finite (otherwise either $(\omega^*+\omega) \cdot \omega$ or
$(\omega^*+\omega) \cdot \omega^*$ would be embeddable into some
$\equiv_{\alpha'}$-equivalence class). Since each $\equiv_{\alpha'}$-
equivalence class belongs to $\Sigma(U)$ we may rewrite $C$ as a sum,
indexed by $\omega$, $\omega^*$ or $\omega^*+\omega$, of members of
$Ind(U)$. According to (2) of Lemma \ref{lem:wqo}, $C$ has a
decomposition as above. \\
\noindent Case 2. $\alpha$ is a limit ordinal.  In this
case, from (2) of Lemma \ref{lem:wqo}, $C$ has a decomposition where
$C_{0}\in Ind(U)$ or $C_{0}$ is a unbounded $\kappa^*$-sum of members
of $Ind(U)$, namely, $C_{0}= \sum^*_{\alpha<\kappa}A_{\alpha}$, and
$C_{n-1}\in Ind(U)$ or $C_{n-1}$ is an unbounded $\lambda$-sum of
members of $Ind(U)$, namely, $C_{n-1}= \sum_{\beta<\lambda}B_{\beta}$,
with $\kappa, \lambda$ regular cardinals. Necessarily,
$\kappa=\lambda=\omega$. Indeed if, for an example $\lambda >\omega$, 
then since neither $(\omega^*+\omega) \cdot \omega$ nor
$(\omega^*+\omega) \cdot \omega^*$ are embeddable into $
\sum_{\beta<\beta'}B_{\beta}$ for $\beta'<\lambda$ we have the same
property for $C_{n-1}$ hence $C_{n-1}\in \Sigma (U)$ and refining the
decomposition we may suppose $C_{n-1}\in Ind(U)$. According to the
weakening of Lemma \ref{lem-0-equiv finie}, each infinite $C_{i}$ is
equimorphic to some $C'_{i}$ such that such that all the
$\equiv_0$-equivalence classes are infinite. Thus replacing those
$C_{i}$ by $C'_{i}$ we get an equimorphic chain $C'$ having only
finitely many finite $\equiv_0$-equivalence classes, each made of
consecutive $C_{i}$ of size $1$. Let $M$ be maximum of their
size. Either $C'_0$ or $C'_{n-1}$ does not belong to $Ind(U)$. Suppose
$C'_{n-1}\not \in Ind(U)$. Let $\varphi\in \CC $ and
$E_{\varphi(m)}$ be the chain $\omega + \varphi(m) + M + \omega^*$. Set
$C'_{\varphi}$ obtained by substituting $C'_{n-1,
\varphi}=\sum_{m<\omega}(B'_{m}+E_{\varphi(m)})$ to $C'_{n-1,
\varphi}=\sum_{m<\omega}B'_{m}$. Since $C'_{n-1}$ is an unbounded sum
containing $(\omega^*+\omega) \cdot \omega$, $C'_{n-1,\varphi}$ is
equimorphic to $C'_{n-1}$, hence $C'_{\varphi}$ is equimorphic to
$C'$. Furthermore, if $\varphi \not = \varphi'$ then $C'_{\varphi}\not
\simeq C'_{\varphi'}$. Hence $\sib(C')\geq \CC$ as
claimed.  \qed
\end{proof}

We now have all the tools to complete the proof of Theorem
\ref{thm:scattered}.

\begin{proof}[of Theorem \ref{thm:scattered}] 

$(i) \Rightarrow (ii)$. \\
\noindent Suppose that $C$ is scattered and $\sib(C)= \kappa < \CC$. 
According to Propositions \ref{prop:manysibling} and
\ref{prop:finitesum} , $C$ is a finite sum $\sum_{j<m} D_j$ of
surordinals and their reverse, and we may suppose $m$ minimum.  We prove
that in this case $max \{sib (D_j): j<m\}= \sib(C)$, and for this we do
not require that $\sib(C)< \CC$.

\noindent According to Lemma \ref{lem:equivclasses}, each $D_j$ is contained into
some equivalence class $\overline D_j$ of $\equiv_{well}$ or of
$\equiv_{well^*}$. We can write $\overline D_j= U+D_j+V$, and we claim
that $\sib(D_j)=\sib(\overline D_j)$. Indeed, if $U$ and $V$ are
finite, this follows from the computation of the siblings in
Proposition \ref{prop:siblingsurordinal}. If $U$ is infinite then since $m$
is minimum, $U=D_{j-1}\cap D_j$, hence $U\subset \overline D_{j-1}\cap
\overline D_j$; furthermore Lemma \ref{lem:meet} applies and hence
$\overline D_{j-1}\cap \overline D_j$ has order type
$\omega^*+\omega$. From this, it follows that $\overline D_j$ has
order type $\omega^*+\omega +Ê\lambda$ for some ordinal $\lambda$ and
furthermore $V$ is finite (apply Lemma \ref{lem:meet}).  Thus in this
case $\sib(D_j)=\sib(\overline D_j)=1$. Since $\sib(\overline D_j)\leq
sib (C)$ by Lemma \ref{lem:numbersibling}, we have $\sib( D_j)\leq sib
(C)$. Setting $\kappa'= max \{sib (D_j), j<m\}$ we obtain $\kappa'\leq
\sib(C)$.
 
\noindent We prove that $\kappa'= \kappa$. If $\kappa=1$,
then since $\kappa'\leq \kappa$ the property holds. We may suppose
$\kappa>1$. Since each $D_j$ is a surordinal or the reverse of a
surordinal, it has a decomposition as a finite sum of indecomposable
chains, hence $C$ has a decomposition as a finite sum $C= \sum_{i<n}
C_i$ where the $C_i$'s are indecomposable, and again we may suppose
$n$ minimum. Since these indecomposables are either strictly right or
strictly left indecomposable, Formula \ref{formula1} from Corollary
\ref{cor:comput} applies. Each $C_i$ is a surordinal or the reverse of
a surordinal, (indeed since $C_i\subseteq \sum_{j<m} D_j$ and $C_i$ is
indecomposable, $C_i$ is embeddable into some $D_j$; since $D_j$ is
either a surordinal or the reverse of a surordinal, the assertion
follows). It follows from Proposition \ref{prop:siblingsurordinal}
that $\sib(C_i)$ and $\sib(C_i+C_{i+1})$ are $1$ or infinite, thus
$\sib(C)= max\{ \sib(C_i): i\not \in K, \sib(C_i+C_{i+1}): i\in K\}$. If
$sub(C_i)= \kappa$, let $j<m$ such that $C_i\setminus D_j<C_i$. In
this case $sub(C_i\cap D_j)= \kappa$, and hence $sub(D_j)=\kappa$. If
$i\in K$ and $\sib(C_i+C_{i+1})=\kappa$, then since $\kappa>1$,
$\kappa=max\{\sib(C_i), \sib(C_{i+1})\}$. There is some $i<m$ such that
$\sib((C_i+C_{i+1})\cap D_j)=\kappa$ (note that if $C_i\setminus
D_j<C_i$, we have $sub(C_i \cap D_j)=sub (C_i$)), hence
$\sib(D_j)=\kappa$.

$(ii) \Rightarrow(i)$.  Since $C$ is a finite sum of surordinal and
reverse of surordinals, $C$ is scattered. The fact that
$\sib(C)=\kappa$ follows from that fact that $max \{sib (D_j), j<m\}=
\sib(C)$ proved just a above. \qed
\end{proof}

We now provide the argument for Corollary \ref{cor:corollary2}.

\begin{proof} [of Corollary \ref{cor:corollary2}]

Note first that part (2) follows from part (1), thus let us prove that (1) holds. 

Suppose that $C$ is scattered and $\sib(C)=\kappa < \CC$. According to
Theorem 2, $C$ is a finite sum $\sum_{j<m} D_j$ of surordinals and
their reverse and, $m$ being minimum, $\kappa= max \{\sib(D_j):j<m\}$.
According to Proposition
\ref{prop:siblingsurordinal}, each $D_j$ or its reverse is either an
ordinal, a surordinal of the form $\omega^{\alpha} \cdot
\omega^*+\gamma$ with $\gamma<\omega^{\alpha+1}$ if $D_j$ is pure, and
$\omega^{\alpha} \cdot \omega^*+\omega^{\beta}+\gamma$ with $\alpha+1
\leq \beta$ otherwise. 

\noindent We may write $D_j= D_{j,0}+D_{j,1}$ with
$D_{j,0}=\omega^{\alpha} \cdot \omega^*$ and $D_{j,1}=\gamma$ if $D_j$
is pure (and $\alpha\not =0$), and $D_{j,0}=\omega^{\alpha} \cdot
\omega^*+\omega^{\beta}$ and $D_{j,1}\gamma$ if $D_j$ is not
pure. From this we obtain a decomposition as stated in the
corollary. 

\noindent In any other decomposition, the number of components of the
form $\omega^{\alpha} \cdot \omega^*$ or its reverse cannot be smaller
(otherwise, in order to be eliminated in such other decomposition, a
component $D_{j,0}=\omega^{\alpha} \cdot \omega^*$ will appear in a
surordinal of the form $\omega^{\alpha} \cdot
\omega^*+\omega^{\beta}+\gamma$ with $\alpha+1
\leq \beta$. Due to the minimality of $m$, this surordinal must be an
initial segment of $D_j$, which is impossible).  Since
$\sib(D_{j,0})=\sib(D_j)$, $\kappa$ is the maximum of the cardinality of
the pure $D_{j,0}$ (distinct from $\omega^*$) and their reverse.  

Now, conversely, suppose that $C$ has a decomposition $\sum_{i<n} C_i$
as stated. Then, since the members of this decomposition are
surordinals and their reverse, $C$ is scattered and according to
Theorem 2, $C$ is a finite sum $\sum_{j<m} D_j$ of surordinals and
their reverse and, $m$ being minimum, $\kappa= max
\{\sib(D_j):j<m\}$. Since $\omega^{\alpha} \cdot \omega^*$ is indecomposable,
each $C_i$ of this form is embeddable into some $D_j$. If $\alpha\geq
1$,the minimality of $n$ ensures that $D_j=\omega^{\alpha} \cdot
\omega^*+\gamma$ with $\gamma<\omega^{\alpha+1}$, hence $\sib(D_j)=
\vert  \omega^{\alpha} \cdot \omega^*\vert$. Since for the other $D_j$'s or
their reverse, we have $\sib(D_j)=1$, we obtain that $\kappa$ is the
maximum of the cardinality of the $C_i$ of the form
$\omega^{\alpha} \cdot \omega^*$ (with $\alpha\geq 1$) or its reverse. \qed
\end{proof}

The next argument proves the full characterization of chains with a
small number of siblings.

\begin{proof} [of Theorem \ref{thm:fullchar}]

We may assume that $C$ is non-scattered, and write $C=\sum_{i \in D}
C_i$ where each $C_i$ is scattered and $D$ is dense and infinite.

First if $\sib(C)=\kappa < \CC$, then by Lemma \ref{lem:interval} any
embedding must preserve each $C_i$. But now Lemma \ref{lem:sib}
immediately yields the remaining  properties. 

Conversely, assume that $C$ has the prescribed decomposition. Then for
a given $C' \equiv C$ we may assume that $C'\subseteq C$, and thus we
can define $C'_i= \ C_i \cap C'$ which must then be scattered. Moreover
any embedding $f:C \rightarrow C'$ must by assumption satisfy $f(C_i)
\subseteq C'_i$ and thus each $C_i' \equiv C_i$. This immediately
gives $\sib(C) \leq \kappa$. Since clearly we also have $\sib(C)\geq
\kappa$, then $\sib(C)=\kappa$ as desired. \qed
\end{proof}

Finally it remains to complete the proof of Proposition
\ref{prop:allctble}.

\begin{proof} [of Proposition \ref{prop:allctble}]
 
Let  $C = \sum_{i \in D} C_i$ where each $C_i$ is scattered and $D$ is a
countably infinite dense chain, we must show that $\sib(C) \geq \CC$.

\noindent Since $D$ is dense, each $C_i$ is an
$\equiv_{h(C)}$-class and we may apply Lemma 2.  If some $C_i$ has
$\CC$ sibling or infinitely many $C_i$ have more than one siblings
then $C$ has $\CC$ siblings. Thus we may suppose that each $C_i$
except finitely many have one sibling, those exceptions having less
than $\CC$ siblings. According to Theorem 2, each $C_i$ is a finite
sum of surordinals and reverse surordinals, thus as proved by
Jullien\cite {key-J1} the set $Q:= \{C_i: i\in D\}$ is well quasi
ordered.  Let $Q$ be the collection of initial segments of $Q$, that
is $I(Q)= \{I\subseteq Q: C\in I , C'\in Q \mbox{ and } C'\leq
C\Rightarrow C'\in I\}$. Then by a result of Higman, $Q$ is well quasi
ordered exactly when $I(Q)$ is well founded. Thus, in our case $I(Q)$
is well founded.

Now for an interval $J$ of $D$ with at least two elements, define:
\[ I(J) = \{C_i \in Q: \; \exists j \in J \; \mbox{ such that } \;  C_i \leq C_j \} \]
Since each $I(J)$ is an initial segment of $Q$, we can choose an interval $J$
so that $I(J)$ is minimal under inclusion among all such collections
of the form $I(J')$.

\noindent We shall now produce a non-trivial order preserving map $h:J
\rightarrow J$ such that $C_x \leq C_{h(x)}$ for all $x \in J$.
Since such a map $h$ can be extended to the identity on $D$ outside
$J$, it immediately yields a map $f$ as in by Lemma \ref{lem:interval}, 
giving $\sib(C) \geq \CC$ as desired.

To construct $h$, first (using the fact that $Q$ is well
quasi-ordered) choose $x \neq y \in J$ such that $C_x \leq C_y$ and
define $h(x)=y$. Since $D$ and thus $J$ is countable, it suffice to show that given
$x \in J$ and $h$ defined on a finite set of $J$ can always be extended to
$x$. 

So assume that $h$ is defined on $x_0<x_1< \cdots x_{n-1}$, and let $x
\in A_i=(x_{i-1},x_i)$. If $A_i'=(h(x_{i-1}),h(x_i))$, then
$I(A_i)=I(J)=I(A_i')$ by minimality, so there must be $y \in A_i'$
such that $C_x \leq C_y$, and simply define $h(x)=y$.

This completes the proof. \qed
\end{proof}

\section{Conclusion} \label{sec:conc}

In this paper we studied equimorphy in the natural case of chains,
providing some structure results for those having a small number of
siblings.  Our study was motivated by the following tree alternative
conjecture of Bonato and Tardif \cite{key-BT}: for every tree $T$, the
number of trees (counted up to isomorphy) which are equimorphic to T,
is one or infinite. Partial results were obtained so far by Tyomkim
\cite{key-T} and extended to graphs by Bonato et al. \cite{key-BBDS}.
They asked if for every (connected) undirected graph G the number of
(connected) graphs (counted up to isomorphy) which are equimorphic to
G, is one or infinite. Notice that if one considers connected graphs
with loops the conjecture is false. Indeed consider the following
undirected graph $G$ with loops.

\begin{center}
\includegraphics[width=4in]{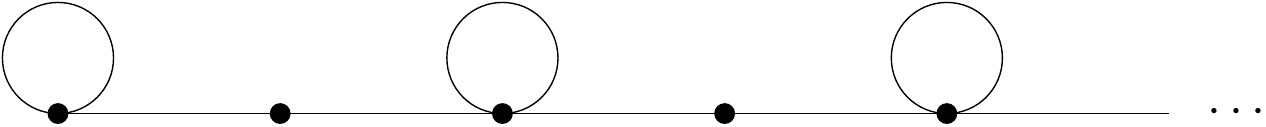}
\end{center}

\noindent One can easily verify that in this case $\sib(G)=2$, 
with the following graph its only non-isomorphic sibling:

\begin{center}
\includegraphics[width=4in]{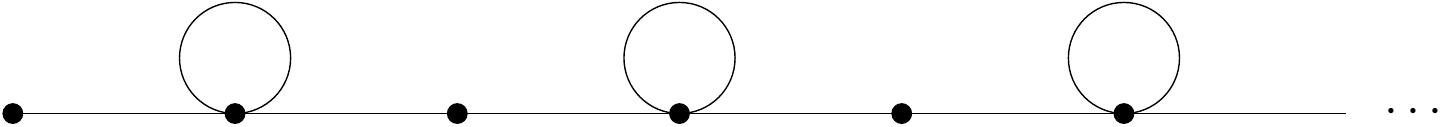}
\end{center}

This is also the case for connected posets, as we may simply consider
a one way infinite fence, which has two equimorphic siblings:

\begin{center}
\includegraphics[width=4in]{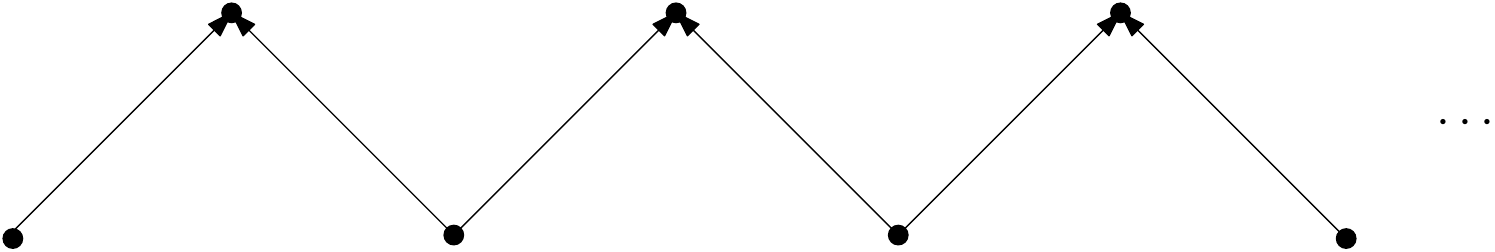}
\end{center}

Thus a complete understanding of the problem not only for trees and posets but
for the general case of a relational structure remains very
interesting.

\bigskip

The first author warmly thanks the Logic group and their staff at the
\emph{Institut Camille Jordan} of Universit\'e Lyon I for their
wonderful hospitality during the final preparation of this work.  The
second author thanks the Department of Mathematics \& Statistics of
the University of Calgary where this research started in the summer of
2012 and for the stimulating atmosphere and support. All three authors
thank the referee for valuable comments.

\end{document}